\documentclass[11pt]{article}
\usepackage{amsmath, amsthm, amssymb, url}
\usepackage[margin=2.5cm]{geometry}
\usepackage{multirow}
\usepackage{tikz, tkz-graph, mathrsfs, enumerate}
\usepackage{authblk}

\usetikzlibrary{arrows}

\newcommand{\id}{\mathrm{id}}

\newcommand{\Rank}{\mathrm{Rank}}

\newcommand{\dd}{\mathrm{d}}

\newcommand{\Aut}{\mathrm{Aut}}
\newcommand{\End}{\mathrm{End}}
\newcommand{\SRank}{\mathrm{SRank}}

\theoremstyle{plain}

\newtheorem{corollary}{Corollary}
\newtheorem{lemma}{Lemma}
\newtheorem{proposition}{Proposition}
\newtheorem{theorem}{Theorem}

\theoremstyle{definition}
\newtheorem{definition}{Definition}

\newtheorem{example}{Example}
\newtheorem{remark}{Remark}

\begin{document}

\title{On the minimal number of generators of endomorphism monoids of full shifts}
\author{Alonso Castillo-Ramirez\footnote{Email: alonso.castillor@academicos.udg.mx}  \\
\small{Department of Mathematics, University Centre of Exact Sciences and Engineering,\\ University of Guadalajara, Guadalajara, Mexico.} }

\maketitle

\begin{abstract}
For a group $G$ and a finite set $A$, denote by $\End(A^G)$ the monoid of all continuous shift commuting self-maps of $A^G$ and by $\Aut(A^G)$ its group of units. We study the minimal cardinality of a generating set, known as the \emph{rank}, of $\End(A^G)$ and $\Aut(A^G)$. In the first part, when $G$ is a finite group, we give upper and lower bounds for the rank of $\Aut(A^G)$ in terms of the number of conjugacy classes of subgroups of $G$. In the second part, we apply our bounds to show that if $G$ has an infinite descending chain of normal subgroups of finite index, then $\End(A^G)$ is not finitely generated; such is the case for wide classes of infinite groups, such as infinite residually finite or infinite locally graded groups.   \\


\end{abstract}

\section{Introduction}\label{intro}

Let $G$ be a group and $A$ a finite set. The \emph{full shift} $A^G$ is the set of all maps $x : G \to A$, equipped with the \emph{shift action} of $G$ on $A^G$:
\[ (g \cdot x) (h) = x(g^{-1}h), \ \ \forall g,h \in G, x \in A^G. \]
We endow $A^G$ with the \emph{prodiscrete topology}, which is the product topology of the discrete topology on $A$. An \emph{endomorphism of $A^G$} is a continuous shift commuting self-map of $A^G$. These are fundamental objects in symbolic dynamics, and by Curtis-Hedlund Theorem (see \cite[Theorem 1.8.1]{CSC10}) a map $\tau : A^G \to A^G$ is an endomorphism if and only if it is a \emph{cellular automaton} of $A^G$ (i.e. there is a finite subset $S \subseteq G$, called a \emph{memory set}, and a function $\mu : A^S \to A$ satisfying $\tau(x)(g) = \mu (( g^{-1} \cdot x) \vert_{S})$, $\forall x \in A^G, g \in G$).

Equipped with composition of functions, the set $\End(A^G)$ of endomorphisms of $A^G$ is a monoid. The \emph{group of units} (i.e. group of invertible elements) of $\End(A^G)$ is denoted by $\Aut(A^G)$. When $\vert A \vert \geq 2$ and $G = \mathbb{Z}$, several interesting properties are known for $\Aut(A^\mathbb{Z})$: it is a countable group that is not finitely generated and it contains an isomorphic copy of every finite group, as well as the free group on a countable number of generators (see \cite{BLR88} and \cite[Sec. 13.2]{LM95}). However, despite of several efforts, most of the algebraic properties of $\End(A^G)$ and $\Aut(A^G)$ still remain unknown. 

Given a subset $T$ of a monoid $M$, the \emph{submonoid generated} by $T$, denoted by $\langle T \rangle$, is the smallest submonoid of $M$ that contains $T$; this is equivalent as defining $\langle T \rangle := \{ t_1 t_2 \dots t_k \in M : t_i \in T, \ \forall i, \ k \geq 0 \}$. We say that $T$ is a \emph{generating set of $M$} if $M= \langle T \rangle$. The monoid $M$ is said to be \emph{finitely generated} if it has a finite generating set. The \emph{rank} of $M$ is the minimal cardinality of a generating set:
\[ \Rank(M) := \min\{\vert T \vert : M= \langle T \rangle \}. \]	

The question of finding the rank of a monoid is important in semigroup theory; it has been answered for several kinds of transformation monoids and Rees matrix semigroups (e.g., see \cite{AS09,G14}). For the case of monoids of endomorphisms of full shifts over finite groups, the question has been addressed in \cite{CRG16b,CRG17, CRSA19}; in particular, the rank of $\Aut(A^G)$ when $G$ is a finite cyclic group has been examined in detail in \cite{CRG16a}. 

In this paper, we study the rank of $\Aut(A^G)$ and $\End(A^G)$. In Section \ref{basic}, we introduce notation and review some basic facts on group theory and the Rank function. In Section \ref{finite}, when $G$ is a finite group, we use the structure theorem for $\Aut(A^G)$ obtained in \cite{CRG17} to provide upper and lower bounds for the rank of $\Aut(A^G)$ in terms of the number of conjugacy classes of subgroups of $G$. We specialize in some particular cases such as cyclic, dihedral, Dedekind, and permutation groups. Finally, in Section \ref{infinite}, we apply our bounds to provide an elementary proof of the following theorem. 

\begin{theorem}\label{main}
Let $A$ be a finite set, and let $G$ be a group has an infinite descending chain of normal subgroups of finite index in $G$. Then, $\End(A^G)$ is not finitely generated.
\end{theorem}

This theorem implies that $\End(A^G)$ is not finitely generated for wide classes of infinite groups, such as infinite residually finite or infinite locally graded groups. However, it does not cover the cases of infinite groups with few normal subgroups of finite index, such as infinite symmetric groups or infinite simple groups. 

This paper is an extended version of \cite{CRSA19}. All sections have been restructured, the exposition has been improved, and the results of Section \ref{infinite} have been greatly extended. Moreover, the technical proof of Theorem \ref{dihedral}, about dihedral groups, has been omitted.


\section{Basic Results} \label{basic}

We assume the reader has certain familiarity with basic concepts of group theory. 

Let $G$ be a group and $A$ a finite set. The \emph{stabiliser} and \emph{$G$-orbit} of a configuration $x \in A^G$ are defined, respectively, by
\[ G_x := \{ g \in G : g \cdot x = x \} \text{ and } Gx := \{ g \cdot x : g \in G \}. \]
Stabilisers are subgroups of $G$, while the set of $G$-orbits forms a partition of $A^G$. 

Two subgroups $H_1$ and $H_2$ of $G$ are \emph{conjugate} in $G$ if there exists $g \in G$ such that $g^{-1} H_1 g = H_2$. This defines an equivalence relation on the subgroups of $G$. Denote by $[H]$ the conjugacy class of $H \leq G$. A subgroup $H \leq G$ is \emph{normal} if $[H] = \{ H \}$ (i.e. $g^{-1}H g = H$ for all $g \in G$). Let $N_G(H) := \{ g \in G : H = g^{-1} H g \} \leq G$ be the \emph{normaliser of $H$ in $G$}. Note that $H$ is always a normal subgroup of $N_G(H)$. Denote by $r(G)$ the total number of conjugacy classes of subgroups of $G$, and by $r_i(G)$ the number of conjugacy classes $[H]$ such that $H$ has index $i$ in $G$: 
\begin{align*}
r(G) & := \vert \{ [H] : H \leq G \} \vert, \\
r_i(G) & := \vert \{ [ H] : H \leq G, \ [G:H] = i \} \vert.
\end{align*}

For any $H \leq G$, denote
\[ \alpha_{[H]}(G ; A):= \vert \{ Gx \subseteq A^G :  [G_x] = [H]  \} \vert. \]
This number may be calculated using the Mobius function of the subgroup lattice of $G$, as shown in \cite[Sec. 4]{CRG17}.

For any integer $\alpha \geq 1$, let $S_{\alpha}$ be the symmetric group of degree $\alpha$. The \emph{wreath product} of a group $C$ by $S_\alpha$ is the set
\[ C \wr S_{\alpha} := \{ (v; \phi) : v \in C ^\alpha, \phi \in S_\alpha \} \]
equipped with the operation $(v;\phi) \cdot (w; \psi) = ( v  \cdot w^{\phi}; \phi \psi)$, for any $v,w \in C^\alpha, \phi, \psi \in S_\alpha$, where $\phi$ acts on $w$ by permuting its coordinates:
\[  w^\phi = (w_1, w_2, \dots, w_\alpha)^\phi := (w_{\phi(1)}, w_{\phi(2)}, \dots, w_{\phi(\alpha)}). \]
In fact, as may be seen from the above definitions, $C \wr S_{\alpha}$ is equal to the external semidirect product $C^{\alpha} \rtimes_{\varphi} S_{\alpha}$, where $\varphi : S_\alpha \to \Aut(C^{\alpha})$ is the action of $S_{\alpha}$ of permuting the coordinates of $C^{\alpha}$. For a more detailed description of the wreath product see \cite{AS09}.

The $\Rank$ function on monoids does not behave well when taking submonoids or subgroups: in other words, if $N$ is a submonoid of $M$, there may be no relation between $\Rank(N)$ and $\Rank(M)$. For example, if $M=S_n$ is the symmetric group of degree $n \geq 3$ and $N$ is a subgroup of $S_n$ generated by $\lfloor \frac{n}{2} \rfloor$ commuting transpositions, then $\Rank(S_n) = 2$, as $S_n$ may be generated by a transposition and an $n$-cycle, but $\Rank(N) = \lfloor \frac{n}{2} \rfloor$. It is even possible that $M$ is finitely generated but $N$ is not finitely generated (such as the case of the free group on two symbols and its commutator subgroup). However, there are some tools that we may use to bound the rank.

For any subset $U$ of a monoid $M$, the \emph{relative rank} of $U$ in $M$ is
\[ \Rank(M:U) = \min \{ \vert W \vert : M = \langle U \cup W \rangle  \}. \]
When $M$ is a finite monoid and $U$ is the group of units of $M$, we have the basic identity
\begin{equation}\label{rank-formula}
 \Rank(M) = \Rank(M:U) + \Rank(U),
\end{equation}
which follows as any generating set for $M$ must contain a generating set for $U$ (see \cite[Lemma 3.1]{AS09}). The relative rank of $\Aut(A^G)$ in $\End(A^G)$ has been established in \cite[Theorem 7]{CRG17} for finite \emph{Dedekind groups} (i.e. groups in which all subgroups are normal).   

\begin{lemma}\label{le:group of units}
For any finite group $G$ and finite set $A$,
\[ \Rank(\Aut(A^G)) \leq \Rank(\End(A^G)). \]
\end{lemma}
\begin{proof}
As $G$ and $A$ are both finite, $\End(A^G)$ is a finite monoid. The result follows by (\ref{rank-formula}). 
\end{proof}

\begin{lemma}\label{le:basic}
Let $G$ and $H$ be a groups, and let $N$ be a normal subgroup of $G$. Then:
\begin{enumerate}
\item $\Rank(G/N) \leq \Rank(G)$.
\item $\Rank(G \times H) \leq \Rank(G) + \Rank(H)$.
\item $\Rank( G \wr S_\alpha) \leq \Rank(G) + \Rank(S_\alpha)$, for any $\alpha \geq 1$.
\item $\Rank(\mathbb{Z}_d \wr S_\alpha) = 2$, for any $d, \alpha \geq 2$. . 
\end{enumerate}
\end{lemma}
\begin{proof}
Parts 1 and 2 are straightforward. For parts 3 and 4, see \cite[Corollary 5]{CRG17} and \cite[Lemma 5]{CRG16a}, respectively.
\end{proof}


\section{Finite groups} \label{finite}

The main tool of this section is the following structure theorem for $\Aut(A^G)$.

\begin{theorem}[\cite{CRG17}] \label{th:ICA}
Let $G$ be a finite group and $A$ a finite set of size $q \geq 2$. Let $[H_1], \dots, [H_r]$ be the list of all different conjugacy classes of subgroups of $G$. Let $\alpha_i :=\alpha_{[H_i]}(G ; A)$. Then,
\[ \Aut(A^G) \cong \prod_{i=1}^{r} \left( (N_{G}(H_i)/H_i) \wr S_{\alpha_i} \right). \]
\end{theorem}

Because of part 3 in Lemma \ref{le:basic}, it is now relevant to determine some values of the $\alpha_i$'s that appear in Theorem \ref{th:ICA}.

\begin{lemma} \label{alpha}
Let $G$ be a finite group and $A$ a finite set of size $q \geq 2$. Let $H$ be a subgroup of $G$.
\begin{enumerate}
\item $\alpha_{[G]}(G;A) = q$. 
\item $\alpha_{[H]}(G;A) = 1$ if and only if $[G : H] = 2$ and $q=2$.
\item If $q \geq 3$, then $\alpha_{[H]}(G;A)  \geq 3$. 
\end{enumerate} 
\end{lemma}
\begin{proof}
Parts 1 and 2 correspond to Remark 1 and Lemma 5 in \cite{CRG17}, respectively. For part 2, Suppose that $q \geq 3$ and $\{0,1,2 \} \subseteq A$. Define configurations $z_1, z_2, z_3 \in A^G$ as follows,
\[ 
z_1 (g)  = \begin{cases}
1 & \text{if } g \in H \\
0 & \text{if } g \not\in H, 
\end{cases} \ \ 
z_2 (g)  = \begin{cases}
2 & \text{if } g \in H \\
0 & \text{if } g \not\in H, 
\end{cases} \ \ 
z_3 (g)  = \begin{cases}
1 & \text{if } g \in H \\
2 & \text{if } g \not\in H, 
\end{cases} \ \ 
\]
All three configurations are in different orbits and  $G_{z_i} = H$, for $i=1,2,3$. Hence $\alpha_{[H]}(G;A) \geq 3$. 
\end{proof}

Although we shall not use explicitly part 3 of the previous lemma, the result is interesting as it shows that, for $q \geq 3$, our upper bounds cannot be refined by a more careful examination of the values of the $\alpha_i's$, as, for all $\alpha \geq 3$, we have $\Rank(S_\alpha) = 2$.


\subsection{Cyclic and dihedral groups}

The rank of $\Aut(A^G)$ when $G$ is a finite cyclic group has been examined in detail in \cite{CRG16a}. Let $\dd(n)$ be number of divisors of $n$, including $1$ and $n$ itself. Let $\dd_-(n)$ and $\dd_+(n)$ be the number of odd and even divisors of $n$, respectively.

\begin{theorem}[Theorem 4 in \cite{CRG16a}]
Let $n \geq 2$ be an integer and $A$ a finite set of size $q \geq 2$.
\begin{description}
\item[(i)] If $n$ is not a power of $2$, then
\[ \Rank( \Aut(A^{\mathbb{Z}_n})  ) = \begin{cases}
\dd(n) + \dd_+(n)  - 1 + \epsilon(n,2) & \text{if } q=2 \text{ and } n \in 2\mathbb{Z}; \\ 
\dd(n) + \dd_+(n) + \epsilon(n,q), & \text{otherwise;}
\end{cases} \]
where $0 \leq \epsilon(n,q) \leq \dd(n) - \dd_+(n) - 2$. 
\item[(ii)] If $n = 2^k$, then
\[ \Rank( \Aut(A^{\mathbb{Z}_{2^k}})  )  = \begin{cases}
2 \dd(2^k) - 2 = 2k & \text{if } q=2; \\
2 \dd(2^k) -1 = 2k + 1 & \text{if } q \geq 3.
\end{cases} \]
\end{description}
\end{theorem}

\begin{example}
If $p$ is an odd prime, the previous theorem implies that
\[ \Rank( \Aut(A^{\mathbb{Z}_p})  ) = 2. \]

\end{example}

We now turn our attention to the dihedral group $D_{2n}$ of order $2n$.

\begin{theorem}[\cite{CRSA19}] \label{dihedral}
Let $n \geq 3$ be an integer and $A$ a finite set of size at least $2$. 
\[  \Rank(\Aut(A^{D_{2n}})) =
\begin{cases}
  \dd_-(2n) + 2\dd_+(2n) - 1 + \epsilon_1 &  \text{if $n$ is odd and $q = 2$,} \\
  \dd_-(2n) + 2\dd_+(2n) + \epsilon_2 & \text{if $n$ is odd and $q \geq 3$,} \\
     \dd_-(2n) + 2\dd_+(2n) + 2 \dd_+(n) - 1 + \epsilon_1   & \text{if $n$ is even and $q = 2$,}\\
     \dd_-(2n) + 2\dd_+(2n) + 4 \dd_+(n)  + \epsilon_2 & \text{if $n$ is even and $q \geq 3$,}
\end{cases} \]
where
\begin{align*}
0 & \leq \epsilon_1 \leq \dd(2n) - 2, \\
 0 & \leq \epsilon_2 \leq \dd(2n) - 1.
\end{align*}
\end{theorem}

\begin{example}
Let $A$ be a finite set of size $q \geq 2$. By the previous theorem,
\begin{align*}
5 \leq \Rank(\Aut(A^{D_6})) \leq  7 \quad & \text{ if } q = 2, \\
6  \leq \Rank(\Aut(A^{D_6})) \leq 9  \quad & \text{ if } q \geq 3.
\end{align*}
On the other hand,
\begin{align*}
  10 \leq \Rank(\Aut(A^{D_8})) \leq 12 \quad & \text{ if } q = 2, \\
15 \leq \Rank(\Aut(A^{D_8})) \leq 18 \quad & \text{ if } q \geq 3.  
\end{align*}
\end{example}


\subsection{Dedekind groups}

Recall that $r(G)$ denotes the total number of conjugacy classes of subgroups of $G$ and $r_i(G)$ the number of conjugacy classes $[H]$ such that $H$ has index $i$ in $G$. The following results are an improvement of \cite[Corollary 5]{CRG17}.

\begin{theorem}\label{cor:bound}
Let $G$ be a finite Dedekind group and $A$ a finite set of size $q \geq 2$. Let $r:=r(G)$ and $r_i := r_i(G)$. Let $p_1, \dots, p_s$ be the prime divisors of $\vert G \vert$ and define $r_P :=\sum_{i=1}^s r_{p_i}$. Then,
\[ \Rank(\Aut(A^G))  \leq \begin{cases}
(r - r_P  - 1) \Rank(G) + 2 r - r_2 - 1,  & \text{ if } q=2, \\
(r - r_P  - 1) \Rank(G) + 2 r, & \text{ if } q \geq 3.
\end{cases} \]
\end{theorem}
\begin{proof}
Let $H_1, H_2, \dots, H_r$ be the list of different subgroups of $G$ with $H_r = G$. If $H_i$ is a subgroup of index $p_k$, then $(G/H_i)\wr S_{\alpha_i} \cong  \mathbb{Z}_{p_k} \wr S_{\alpha_i}$ is a group with rank $2$, by Lemma \ref{le:basic}. Thus, by Theorem \ref{th:ICA} we have:
\begin{align*}
\Rank(\Aut(A^G)) & \leq  \sum_{i=1}^{r-1} \Rank((G/H_i)\wr S_{\alpha_i}) + \Rank(S_q)  \\ 
& \leq \sum_{[G:H_i] = p_k} 2 + \sum_{[G:H_i] \neq p_k } (\Rank(G) + 2) + 2\\
& =  2 r_P +  (r - r_P - 1 )(\Rank(G) + 2) + 2 \\ 
& =  (r - r_P  - 1) \Rank(G) + 2r .  
\end{align*}
If $q=2$, we may improve this bound by using Lemma \ref{alpha}:
\begin{align*}
\Rank(\Aut(A^G)) & \leq \sum_{[G : H_i]=2}\Rank((G/H_i)\wr S_1) + \sum_{[G : H_i] = p_k \neq 2}\Rank((G/H_i)\wr S_{\alpha_i}) \\ 
& + \sum_{1 \neq [G:H_i] \neq p_k} \Rank((G/H_i)\wr S_{\alpha_i}) + \Rank(S_2) \\  
& = r_2 + 2(r_P - r_2) +  (r - r_P - 1 )(\Rank(G) + 2)  + 1   \\
& =   (r - r_P  - 1) \Rank(G) + 2r - r_2 - 1.  
\end{align*}

 \end{proof}

\begin{example}
The smallest example of a nonabelian Dedekind group is the quaternion group 
\[ Q_8 = \langle x,y \; \vert \; x^4 = x^2 y^{-2} = y^{-1}xy x  = \id  \rangle, \]
which has order $8$. It is generated by two elements, and it is noncyclic, so $\Rank(Q_8) = 2$. Moreover, $r = r(Q_8) = 6$ and, as $2$ is the only prime divisor of $8$, we have $r_P  = r_2 = 3$. Therefore,
\[ \Rank(\Aut(A^{Q_8}))  \leq \begin{cases}
(6 - 3 - 1) \cdot 2 + 2 \cdot 6 - 3 - 1 = 12,  & \text{if } q=2, \\
(6 - 3  - 1) \cdot 2 + 2 \cdot 6 = 16, & \text{if } q \geq 3.
\end{cases} \]
\end{example}
 
 \begin{corollary}
 Let $G$ be a finite Dedekind group and $A$ a finite set of size $q \geq 2$. With the notation of Theorem \ref{cor:bound},
\[ \Rank(\End(A^G))  \leq \begin{cases}
(r - r_P  - 1) \Rank(G) + \frac{1}{2} r (r+5) - 2r_2 - 1,  & \text{if } q=2 \\
(r - r_P  - 1) \Rank(G) + \frac{1}{2} r (r+5), & \text{otherwise.}
\end{cases} \]
 \end{corollary}
 \begin{proof}
 The result follows by Theorem \ref{cor:bound}, identity (\ref{rank-formula}) and the basic upper bound for the relative rank that follows from \cite[Theorem 7]{CRG17}:
 \[ \Rank(\End(A^G):\Aut(A^G)) \leq \begin{cases} 
 \binom{r}{2} + r - r_2 & \text{if } q=2 \\
 \binom{r}{2} + r, & \text{otherwise.}
 \end{cases}  \]

 \end{proof}

\subsection{Arbitrary finite groups}

Now focus now when $G$ is not necessarily a Dedekind group. Because of the decomposition of Theorem \ref{th:ICA}, it is relevant to bound the rank of $N_{G}(H)/H$, when $H$ is a subgroup of $G$, in order to bound the rank of $\Aut(A^G)$. When the index of $H$ in $G$ is prime, we may find a tight bound. 

\begin{lemma}\label{le-aux1}
Let $G$ be a finite group and $H$ a subgroup of $G$ of prime index $p$. Let $A$ be a finite set of size $q \geq 2$ and $\alpha := \alpha_{[H]}(G;A)$. Then
\[ \Rank\left( (N_{G}(H)/H) \wr S_{\alpha} \right) \leq \begin{cases}
1 & \text{if } p=2 \text{ and } q=2\\
2 & \text{otherwise}.
\end{cases} \] 
\end{lemma}
\begin{proof}
By Lagrange's theorem, $N_{G}(H)=H$ or $N_{G}(H)=G$. Hence, in order to find an upper bound for the above rank, we assume that $H$ is normal in $G$. As the index is prime, $G/H \cong \mathbb{Z}_p$. If $p=2$ and $q=2$, Lemma \ref{alpha} shows that $\alpha = 1$, so $\Rank(\mathbb{Z}_2 \wr S_{1})  = 1$. For the rest of the cases we have that $\Rank(\mathbb{Z}_p \wr S_{\alpha}) = 2$, by Lemma \ref{le:basic}.
\end{proof}

In general, by Lemma \ref{le:basic}, $\Rank(N_{G}(H)/H) \leq \Rank(N_{G}(H))$. A natural way to bound this for all $H \leq G$ is to use the \emph{subgroup rank} of $G$:
\[ \SRank(G) := \max\{ \Rank(K) : K \leq G \}. \]

\begin{lemma}\label{le-aux2}
Let $G$ be a finite group and $H$ a subgroup of $G$. Let $A$ be a finite set of size $q \geq 2$ and $\alpha := \alpha_{[H]}(G;A)$. Then,
\[ \Rank\left( (N_{G}(H)/H) \wr S_{\alpha} \right) \leq \SRank(G) + 2 \leq \log_2(\vert G \vert) + 2. \] 
\end{lemma}
\begin{proof}
By Lemma \ref{le:basic}, $\Rank\left( (N_{G}(H)/H) \wr S_{\alpha} \right) \leq \Rank(N_{G}(H)) + 2$. Moreover, it is clear that $\Rank(N_{G}(H)) \leq \SRank(G)$ for every subgroup $H \leq G$. The second inequality follows by \cite[Lemma 1.2.2]{LS03}, as $\SRank(G) + 2 \leq \log_2(\vert G \vert)$. 
\end{proof}

As an alternative, instead of using the subgroup rank in Lemma \ref{le-aux2}, we may use the \emph{length} of $G$ (see \cite[Sec. 1.15]{C94}) which is defined as the length $\ell := \ell(G)$ of the longest chain of proper subgroups
\[ 1=G_0 < G_1 < \dots < G_\ell = G. \]
Observe that $\Rank(G) \leq \ell(G)$, as the set $\{ g_i \in G : g_i \in G_i - G_{i-1}, \ i=1,\dots, \ell \}$ (with $G_i$ as the above chain of proper subgroups) generates $G$. As $\ell(K) \leq \ell(G)$, for every $K \leq G$, it follows that $\SRank(G) \leq \ell(G)$. 

The lengths of the symmetric groups are known by \cite{CST89}: $\ell(S_n) = \lceil 3n/2 \rceil - b(n)-1$, where $b(n)$ is the numbers of ones in the base $2$ expansion of $n$. As, $\ell(G) = \ell(N) + \ell(G/N)$ for any normal subgroup $N$ of $G$, the length of a finite group is equal to the sum of the lengths of its compositions factors; hence, the question of calculating the length of all finite groups is reduced to calculating the length of all finite simple groups. Moreover, $\ell(G) \leq \log_2(\vert G \vert)$, by \cite[Lemma 2.2]{CST89}.

\begin{theorem}\label{finite-upper-bound}
Let $G$ be a finite group of size $n$, $r:= r(G)$, and $A$ a finite set of size $q \geq 2$. Let $r_i$ be the number of conjugacy classes of subgroups of $G$ of index $i$. Let $p_1, \dots, p_s$ be the prime divisors of $\vert G \vert$ and let $r_P = \sum_{i=1}^s r_i$. Then:
\[ 
\Rank( \Aut( A^G )) \leq \begin{cases}
 (r-r_P -1) \log_2(\vert G \vert) + 2r - r_2 -1 & \text{if } q=2, \\
(r - r_P - 1) \log_2(\vert G \vert) + 2r  & \text{if } q\geq 3.
\end{cases} \]
\end{theorem}
\begin{proof}
Let $H_1, H_2, \dots, H_r$ be the list of different subgroups of $G$ with $H_r = G$. By Theorem \ref{th:ICA} and Lemmas \ref{le:basic}, \ref{le-aux1}, \ref{le-aux2},
\begin{align*}
\Rank( \Aut(A^G )) & \leq \sum_{i=1}^{r-1} \Rank\left( (N_{G}(H_i)/H_i) \wr S_{\alpha_i} \right) + \Rank(S_q) \\
& \leq \sum_{[G : H_i]=p_k} 2 +  \sum_{1 \neq [G : H_i] \neq p_k} (\log_2(\vert G \vert) + 2) + 2 \\
& = 2 r_P+ (r - r_P - 1) (\log_2(\vert G \vert) + 2) + 2  \\
& =  (r - r_P - 1) \log_2(\vert G \vert)  + 2r .
\end{align*}
When $q=2$, we may improve this bound as follows:
\begin{align*}
\Rank( \Aut(A^G)) & \leq   \sum_{[G : H_i]=2} 1  +  \sum_{[G : H_i]=p_k \neq 2} 2   +  \sum_{1<[G : H_i] \neq p_k} (\log_2(\vert G \vert) + 2) + 1 \\
& = r_2  + 2(r_P - r_2) + (r - r_P - 1) (\log_2(\vert G \vert)+2) + 1  \\
& = (r-r_P -1)\log_2(\vert G \vert) + 2r - r_2 - 1.
\end{align*}

\end{proof}

If $G$ is a subgroup of $S_n$ (i.e. if $G$ is a permutation group), we may find a good upper bound for $\Rank( \Aut(A^G))$ in terms of $n$ by using a theorem by McIver and Neumann.

\begin{proposition}
Suppose that $G \leq S_n$, for some $n >3$. Let $r := r(G)$. Then 
\[  \Rank( \Aut(A^G)) \leq \begin{cases}
 (r -1) \left\lfloor \frac{n}{2} \right\rfloor + 2r - r_2 -1 & \text{if } q=2, \\
(r  - 1) \left\lfloor \frac{n}{2} \right\rfloor  + 2r  & \text{if } q \geq 3,
\end{cases} \]
where $\left\lfloor \frac{n}{2} \right\rfloor$ is the floor function. 
\end{proposition}
\begin{proof}
By \cite{MN87}, for every $n > 3$ and every $K \leq S_n$, $\Rank(K) \leq \lfloor \frac{n}{2} \rfloor$. The rest of the proof is analogous to the proof of Theorem \ref{finite-upper-bound}.
\end{proof}

\begin{example}
Consider the symmetric group $S_4$. In this case it is known that $r=r(S_4) = 11$ and $r_2 = 1$ (as $A_4$ is its only subgroup of index $2$). Therefore,
\[  \Rank( \Aut(A^{S_4} )) \leq \begin{cases}
 (11 -1) \frac{4}{2} + 2 \cdot 11  - 1 -1 = 40  & \text{if } q=2, \\
(11  - 1) \frac{4}{2} + 2 \cdot 11 = 42 & \text{if } q \geq 3.
\end{cases} \]
For sake of comparison, the group $\Aut(\{0,1\}^{S_4} )$ has order $2^{2^{24}}$.
\end{example}

Finally, we find a lower bound for the rank of $\Aut(A^G)$, when $G$ is an arbitrary finite group. 

\begin{proposition}\label{lower-bound}
Let $G$ be a finite group and $A$ a finite set of size $q \geq 2$. Then
\[ \Rank(\Aut(A^G) \geq \begin{cases} 
r(G) -  r_2(G)  & \text{if } q=2, \\
r(G)  & \text{otherwise}. 
\end{cases} . \]
\end{proposition} 
\begin{proof}
Let $[H_1], [H_2], \dots, [H_r]$ be the conjugacy classes of subgroups of $G$, with $r=r(G)$. As long as $\alpha_i > 1$, the factor $(N_{G}(H_i)/H_i) \wr S_{\alpha_i}$, in the decomposition of $\Aut(A^G)$, has a proper normal subgroup $ (N_{G}(H_i)/H_i) \wr A_{\alpha_i}$ (where $A_{\alpha_i}$ is the alternating group of degree $\alpha_i$). We know that $\alpha_i =1$ if and only if $[G:H]=2$ and $q=2$ (Lemma \ref{alpha}). Hence, for $q \geq 3$, we have
\[  \Rank(\Aut(A^G)) \geq \Rank\left( \frac{\prod_{i=1}^{r} \left( (N_{G}(H_i)/H_i) \wr S_{\alpha_i} \right)}{ \prod_{i=1}^{r} \left( (N_{G}(H_i)/H_i) \wr A_{\alpha_i} \right) } \right) 
= \Rank\left(\prod_{i=1}^{r} \mathbb{Z}_2\right) = r.  \]    

Assume now that $q=2$, and let $[H_1], \dots, [H_{r_2}]$ be the conjugacy classes of subgroups of index two, with $r_2 = r_2(G)$. Now, $\Rank(\Aut(A^G))$ is at least
\[ \Rank\left(\frac{\prod_{i=1}^{r} \left( (N_{G}(H_i)/H_i) \wr S_{\alpha_i} \right)}{ \prod_{i=1}^{r} \left( (N_{G}(H_i)/H_i) \wr A_{\alpha_i} \right) } \right) =  \Rank\left(\prod_{i=r_2+1}^r \mathbb{Z}_2 \right) = r - r_2,  \]
and the result follows. 
\end{proof}


\section{Infinite groups} \label{infinite}

Now we turn our attention to the case when $G$ is an infinite group. It was shown in \cite{BLR88} that $\Aut(A^{\mathbb{Z}})$ is not finitely generated by studying its action on periodic configurations. In this section, using elementary techniques, we prove that the monoid $\End(A^G)$ is not finitely generated when $G$ contains an infinite descending chain of normal subgroups of finite index. In particular, this implies that $\End(A^G)$ is not finitely generated when $G$ is infinite residually finite, or infinite locally graded. This illustrates an application of the study of ranks of groups of automorphisms of full shifts over finite groups.

\begin{remark}\label{remark}
Let $G$ be a group that is not finitely generated. Suppose that $\End(A^G)$ has a finite generating set $H=\{ \tau_1, \dots, \tau_k \}$. Let $S_i$ be a memory set for each $\tau_i$. Then $G \neq \langle \cup_{i=1}^k S_i \rangle$, so let $\tau \in \End(A^G)$ be such that its minimal memory set is not contained in $\langle \cup_{i=1}^k S_i \rangle$. As a memory set for the composition $\tau_i \circ \tau_j$ is $S_i S_j = \{ s_i s_j : s_i \in S_i, s_j \in S_j \}$, $\tau$ cannot be in the monoid generated by $H$, contradicting that $H$ is a generating set for $\End(A^G)$. This shows that $\End(A^G)$ is not finitely generated whenever $G$ is not finitely generated.  
\end{remark}

The next result, which holds for an arbitrary group $G$, will be our main tool.

\begin{lemma}\label{th:quotient}
Let $G$ be a group and $A$ a set. For every normal subgroup $N$ of $G$, 
\[ \Rank(\End(A^{G/N}))  \leq \Rank(\End(A^G))  . \] 
\end{lemma}
\begin{proof}
By \cite[Proposition 1.6.2]{CSC10}, there is a monoid epimorphism $\Phi : \End(A^G) \to \End(A^{G/N})$. Hence, the image under $\Phi$ of a generating set for $\End(A^G)$ of minimal size is a generating set for $\End(A^{G/N})$ (not necessarily of minimal size).  
\end{proof}

\begin{theorem}\label{descending}
Let $G$ be a group such that there is an infinite descending chain 
\[ G > N_1 > N_2 > \dots \] 
such that, for all $i \geq 1$, $N_i$ is a normal subgroup of finite index in $G$. Then, $\End(A^G)$ is not finitely generated.
\end{theorem}
\begin{proof}
For any finite group $K$, let $n(K)$ be  the number of normal subgroups of $K$. Not that $n(K) \leq r(K)$, as normal subgroups of $K$ are in bijection with conjugacy classes of subgroups of size $1$. 

For each $i \geq 1$, $N_i$ is a normal subgroup of finite index in $G$, so $G/N_i$ is a finite group. By the Correspondence Theorem (see \cite[Theorem 4.9]{R12}), $n(G/N_i) \geq i+1$, as there are at least $i+1$ intermediate normal subgroups between $G$ and $N_i$ (namely, $G, N_1, \dots, N_i$). If 

By Lemma \ref{th:quotient}, Lemma \ref{le:group of units}, and Proposition \ref{lower-bound}, for all $i \geq 1$, 
\begin{align*}
\Rank(\End(A^G)) & \geq \Rank(\End(A^{G/N_i}))  \\
& \geq \Rank(\Aut(A^{G/N_i})) \\
& \geq  r(G/N_i) - r_2(G/N_i) \\
& \geq n(G/N_i) - r_2(G/N_i)  \\
& \geq (i+1) - 1 = i
\end{align*}
as there is at most one subgroup of index two in the infinite descending chain of normal subgroups of finite index (see \cite[Theorem 3.1.3]{R12}. This implies that $\End(A^G)$ is not finitely generated. 
\end{proof}

\begin{example}
For the infinite cyclic group $\mathbb{Z}$ we may easily find a descending chain of normal subgroups of finite index; for example,
\[ \mathbb{Z} > 2\mathbb{Z} > 4 \mathbb{Z} > \dots > 2^k \mathbb{Z} > \dots \]
where $2^k \mathbb{Z} = \{ 2^k s : s \in \mathbb{Z} \}$. By Theorem \ref{descending}, $\End(A^{\mathbb{Z}})$ is not finitely generated.
\end{example}

\begin{definition}
A group $G$ is \emph{residually finite} if for every non-identity $g \in G$ there is a finite group $K$ and a homomorphism $\phi :G \to K$ such that $\phi(g)$ is not the identity in $K$. 
\end{definition}

We shall use the following characterization of residually finite groups.

\begin{proposition}
A group $G$ is residually finite if and only if for every non-identity $g \in G$ there is a normal subgroup $N$ of finite index in $G$ such that $g \not \in N$ 
\end{proposition}
\begin{proof}
Suppose that $G$ is residually finite and consider a non-identity $g \in G$. Then there is a finite group $K$ and a homomorphism $\phi :G \to K$ such that $\phi(g)$ is not the identity in $K$. Then, the subgroup $N:=\ker(\phi)$ is a normal subgroup of $G$ of finite index (because $G/N$ is isomorphic to a subgroup of $K$, with $K$ a finite group). As $\phi(g)$ is not the identity in $K$, we have $g \not \in N$.  

For the converse, we take the natural homomorphism $\phi : G \to G/N$. The group $G/N$ is finite, as $N$ has finite index, and $\phi(g)$ is not the identity in $G/N$ because $g \not \in G$. The result follows.
\end{proof}

Yet another characterization, which we shall not use explicitly here, is that $G$ is residually finite if and only if it is isomorphic to a subgroup of a direct product of a family of finite groups (see \cite[Chapter 2]{CSC10}).   

Important classes of groups that are residually finite are finite groups, finitely generated abelian groups, profinite groups, free groups, and finitely generated linear groups. Moreover, subgroups, direct products, and direct sums of residually finite groups are residually finite. 

\begin{corollary}
Let $G$ be an infinite residually finite group. Then $\End(A^G)$ is not finitely generated.
\end{corollary}
\begin{proof}
We shall construct an infinite descending chain of normal subgroups of finite index in $G$:
\[ G = N_0 > N_1 > N_2 > \dots  \]
Suppose that the group $N_i$ has been constructed. Clearly, $N_i$ is non-trivial as it is a subgroup of finite index in an infinite group. Take $g \in N_i < G$ with $g \neq e$. Then, there exists a normal subgroup $N_i^\prime$ of finite index in $G$ such that $g \not \in N_i^\prime$. Define $N_{i+1} := N_i \cap N_i^\prime$. Note that $N_{i+1}$ is properly contained in $N_i$, as, otherwise, $N_i = N_i \cap N_i^\prime$ implies that $N_i \subseteq N_i^\prime$, contradicting the existence of $g$. Since intersections of normal subgroups are normal subgroups, and intersections of subgroups of finite index have finite index (see \cite[Lemma 2.5.2]{CSC10}), then $N_{i+1}$ has the required properties.  

The result follows by Theorem \ref{descending}. 
\end{proof}

\begin{corollary}
Let $G$ be an infinite abelian group. Then $\End(A^G)$ is not finitely generated. 
\end{corollary}
\begin{proof}
If $G$ is finitely generated, the result follows by the previous corollary, while if $G$ is not finitely generated the result follows by Remark \ref{remark}.
\end{proof}

A group $G$ is said to be \emph{locally graded} if every finitely generated non-trivial subgroup of $G$ contains a proper subgroup of finite index. This class of groups includes all generalized solvable groups and all residually finite groups \cite{KR94}.

\begin{corollary}
If $G$ is an infinite locally graded group, then $\End(A^G)$ is not finitely generated. 
\end{corollary}
\begin{proof}
If $G$ is not finitely generated, the result follows by Remark \ref{remark}. If $G$ is finitely generated, then \cite[3.1]{OP88} shows that there is an infinite descending chain of normal subgroups of finite index. The result follows by Theorem \ref{descending}. 
\end{proof}

\begin{corollary}
If $G$ is an extension of a group $Q$ satisfying the hypothesis of Theorem \ref{descending}, then $\End(A^G)$ is not finitely generated.
\end{corollary}
\begin{proof}
If $G$ is an extension of $Q$ there exists a normal subgroup $N$ of $G$ such that $G/N \cong Q$. By Lemma \ref{th:quotient}, $\Rank(\End(A^{Q}))  \leq \Rank(\End(A^G))$, so the result follows by Theorem \ref{descending}.
\end{proof}

The techniques of this section seem ineffective to investigate if $\End(A^G)$ is finitely generated when $G$ is a group with few normal subgroups, such as the infinite symmetric group, or infinite simple groups.

A natural and interesting variant of Theorem \ref{descending} would be to prove that the group $\Aut(A^G)$ is not finitely generated when $G$ has an infinite descending chain of normal subgroups of finite index. The main problem here is to find an analogous result of Lemma \ref{th:quotient} for $\Aut(A^G)$. This is a hard question in general, as the monoid epimorphism $\Phi : \End(A^G) \to \End(A^{G/N})$ does not necessarily restrict to an epimorphism of the groups of units, i.e. the homomorphism $\Phi \vert_{\Aut(A^G)} : \Aut(A^G) \to \Aut(A^{G/N})$ may not be surjective.


\end{document}